\documentclass[11pt,draft]{amsart}
\usepackage{amsmath,amsthm,amssymb}
\usepackage[latin1]{inputenc}
\usepackage{enumitem}
\usepackage{float}
\usepackage{url}
\usepackage[all]{xy}
\usepackage{xcolor}
\usepackage{booktabs}
\usepackage{tabularx}
\usepackage{comment}
\usepackage[obeyDraft,textsize=tiny]{todonotes}
\specialcomment{draftnote}{\vspace{1em}\begingroup\ttfamily\footnotesize\color{blue}\begin{minipage}{15cm}}{\end{minipage}\endgroup\vspace{1em}}
\specialcomment{alert}{\vspace{1em}\begingroup\ttfamily\footnotesize\color{blue}\begin{minipage}{15cm}}{\end{minipage}\endgroup\vspace{1em}}

\usepackage[pdftex,final]{hyperref}

\newcommand{\Q}{\mathbb{Q}} 
\newcommand{\C}{\mathbb{C}} 
\newcommand{\Z}{\mathbb{Z}} 
\newcommand{\F}{\mathbb{F}} 

\newcommand{\bs}{\backslash}

\newcommand{\calP}{\mathcal{P}}

\newcommand{\frake}{\mathfrak{e}}

\DeclareMathOperator{\SL}{SL}
\DeclareMathOperator{\GL}{GL}

\DeclareMathOperator{\Mp}{Mp}

\DeclareMathOperator{\tr}{tr}

\DeclareMathOperator{\supp}{supp}

\renewcommand{\labelenumi}{(\theenumi)}

\newtheorem{theorem}{Theorem}[section]
\newtheorem{proposition}[theorem]{Proposition}
\newtheorem{lemma}[theorem]{Lemma}

\newtheorem{algorithm}{Algorithm}
\numberwithin{algorithm}{section}

\theoremstyle{definition}

\newtheorem{remark}[theorem]{Remark}

\newcommand{\QD}[1][]{Q_\Delta\ifthenelse{\equal{#1}{}}{}{\left(#1\right)}}

\DeclareMathOperator{\sym}{sym}

\newcommand{\sig}{\operatorname{sig}}
\newcommand{\zxz}[4]{\begin{pmatrix} #1 & #2 \\ #3 & #4 \end{pmatrix}}
\newcommand{\leg}[2]{\left( \frac{#1}{#2} \right)}
\newcommand{\kzxz}[4]{\left(\begin{smallmatrix} #1 & #2 \\ #3 & #4\end{smallmatrix}\right) }
\newcommand{\kabcd}{\kzxz{a}{b}{c}{d}}
\renewcommand{\H}{\mathbb{H}}

\hypersetup
{
    pdfauthor={Stephan Ehlen and Nils-Peter Skoruppa},
    pdftitle={Computing invariants of the Weil representation},
    pdfkeywords={KEYWORDS}
    pdfsubject={SUBJECT}
    bookmarks=true,
    bookmarksnumbered=true,
    bookmarksopen=false,
    hypertexnames=false,
    breaklinks=true,
    colorlinks=true,
    pagebackref=true,
    hyperindex=true,
    plainpages=false,
}

\newcommand {\fqm}[1] {\mathfrak{#1}}
\renewcommand    {\sym}[1]        {\operatorname{#1}}
\DeclareMathOperator{\card}{card}
\newcommand {\cd}[1] {\card\left(#1\right)}

\renewcommand {\zxz}[4] {\left[
    \begin{smallmatrix}
      {#1}&{#2}\\ {#3}&{#4}
    \end{smallmatrix}
\right]}
\renewcommand {\kzxz}[4] {\zxz {#1}{#2}{#3}{#4}} 
\renewcommand {\leg}[2] {\genfrac(){}{0}{#1}{#2}}
\newcommand{\eps}{\epsilon}
\usepackage{placeins}
\usepackage[most, breakable]{tcolorbox}
\usepackage{multicol}

\numberwithin{equation}{section}

\begin{document}
\title{Computing invariants of the Weil representation}
 
\author{Stephan Ehlen}
\email{stephan.ehlen@math.uni-koeln.de}
\address{University of Cologne, Mathematisches Institut, Weyertal 86-90, D-50931 Cologne, Germany}
\author{Nils-Peter Skoruppa}
\email{nils.skoruppa@gmail.de}
\address{Universität Siegen, Fachbereich Mathematik, Walter-Flex-Str. 3, D-57072 Siegen, Germany}


\begin{abstract}
  We propose an algorithm for computing bases and dimensions of spaces
  of invariants of Weil representations of $\SL_2(\Z)$ associated to
  finite quadratic modules. We prove that these spaces are defined
  over~$\Z$, and that their dimension remains stable if we replace the
  base field by suitable finite prime fields.
\end{abstract}
\maketitle

\section{Introduction}
\label{sec:intro}

Weil representations associated to finite abelian groups $A$ equipped
with a non-degenerate quadratic form $Q$ provide a fundamental tool in
the theory of automorphic forms. They are at the basis of the theory
of automorphic products, the theory of Jacobi forms or Siegel modular
forms of singular and critical weight, and they find also applications
in other disciplines like coding theory or quantum field theory. Of
particular interest among the mentioned applications is the space
$\C[A]^G$ of invariants of the Weil representations of $G=\SL_2(\Z)$
associated to a given finite quadratic module $(A,Q)$. Despite the
importance of $\C[A]^G$ for the indicated applications neither any
explicit closed formula is known for the dimension of $\C[A]^G$ nor
any useful description\footnote{However, if $(A,Q)$ possesses
  a self-dual isotropic subgroup $U$ then the characteristic function
  of $U$ is quickly checked to be an invariant, and one can show that,
  in the case that $(A,Q)$ possesses self-dual subgroups, the
  characteristic functions of the self-dual isotropic subgroups span
  in fact the space $\C[A]^G$ (A proof of this will be given
  in~\cite{Skoruppa}). But an arbitrary finite quadratic module does
  not necessarily possess self-dual isotropic subgroups and still
  admits nonzero invariants if its order is big enough.}
of its elements.

The purpose of the present note is to discuss questions related to the
computation of the dimension and a basis of $\C[A]^G$ for a given
finite quadratic module~$(A,Q)$. In particular, we develop an algorithm
(Algorithm~\ref{alg:algo-2}) for computing a basis of $\C[A]^G$ which we
also implemented and ran successfully in various examples. We
mention two results of this article which might be of independent
interest. First, we prove that $\C[A]^G$ always possesses a basis whose
elements are in $\Z[A]$ (Theorem~\ref{thm:ring-of-definition}). Second, if a
finite prime field $\F_\ell$ contains the $N$th root of unity, where
$N$ is the level of $(A,Q)$, then the Weil representation can also be
defined on $\F_\ell[A]$. We prove that then
$\dim \C[A]^G = \dim \F_\ell[A]^G$ (except for possibly $(N,\ell)=(2,3)$).
Our algorithm has already been used succesfully to compute the dimension of spaces of vector valued
cusp forms of weight $2$ and $3/2$ in \cite{bef-simple}, where a classification
of all lattices of signature $(2,n)$ without obstructions to the existence of weakly holomorphic
modular forms of weight $1-\tfrac{n}{2}$ for the associated Weil representation was given.

The plan of this note is as follows. In Section~\ref{sec:fqm} we
recall the basic definitions and facts from the theory of finite
quadratic modules and its associated Weil representations. In
Section~\ref{sec:invariants} we prove some basic facts about the space
of invariants $\C[A]^G$. Most of the material of this section is
probably known to specialists. However, since it is often difficult to
find suitable references we decided to include this section. To
our knowledge Theorem~\ref{thm:ring-of-definition} is new, which shows that
the space of invariants $\C[A]^G$ is in fact defined over $\Z$. In
Section~\ref{sec:algorithm} we explain our algorithm for computing a
basis for $\C[A]^G$, and we discuss some improvements. In
Section~\ref{sec:reduction-mod-p} we consider the reduction of Weil
representations modulo suitable primes $\ell$ and prove that the
dimension of the space of invariants remains stable under
reduction. This interesting fact can be used to improve the run-time of our algorithm.
Finally, in Section~\ref{sec:tables} we provide tables of dimensions for quadratic
modules of small order.

\section{Finite quadratic modules and Weil representations}
\label{sec:fqm}

A {\em finite quadratic module} (also called a finite quadratic form
or discriminant form in the literature) is a pair $\fqm A=(A,Q)$
consisting of a finite abelian group $A$ together with a
$\Q/\Z$-valued non-degenerate quadratic form $Q$ on $A$.  The bilinear
form corresponding to $Q$ is defined as
\begin{equation*}
  Q(x,y) := Q(x+y) - Q(x) - Q(y)
  .  
\end{equation*}
The quadratic form $Q$ is called non-degenerate if $Q( \cdot, \cdot )$
is non-degenerate, i.e.  if there exists no $x \in A\setminus\{0\}$,
such that $Q(x,y) = 0$ for all $y \in A$.  Two finite quadratic
modules $\fqm A=(A,Q)$ and $\fqm B=(B,R)$ are called isomorphic if
there exists an isomorphism of groups $f:A\rightarrow B$ such that
$Q=R\circ f$.  The theory of finite quadratic modules has a long
history; see e.g.~\cite{Wall-1}, \cite{Wall-2}, \cite{Nikulin} and the
upcoming~\cite{Skoruppa}.

If $\fqm L = (L,\beta)$ is an even lattice, the quadratic form $\beta$
on $L$ induces a $\Q/\Z$-valued quadratic form $Q$ on the discriminant
group $L'/L$ of $\fqm L$.  The pair $\sym{D}_{\fqm L}:=(L'/L,Q)$
defines a finite quadratic module, which we call the
\emph{discriminant module} of~$\fqm L$.  According to
\cite[Thm.~(6)]{Wall-1}, any finite quadratic module can be obtained as the
discriminant module of an even lattice $\fqm L$.  If $\fqm A=(A,Q)$ is
a finite quadratic module and $\fqm L$ a lattice whose discriminant
module is isomorphic to~$\fqm A$, then the difference $b^+ - b^-$ of
the real signature $(b^+, b^-)$ of~$\fqm L$ is already determined
modulo $8$ by~$\fqm A$. Namely, by Milgram's
formula~\cite[p.~127]{Milnor-Husemoeller} one has
\[
\frac{1}{\sqrt{\cd A}} \sum_{x \in A} e\left(Q(x)\right) = e((b^{+}
- b^{-})/8),
\]
where we use $e(z) = e^{2\pi i z}$ for $z\in \C$.  We call
\begin{equation*}
  \sig(\fqm A) := b^{+} - b^{-}\bmod 8  \in \Z/8\Z
\end{equation*}
the \emph{signature of $\fqm A$}. The number
\[
  N = \min\{ n\in \Z_{>0}\mid\; \text{$nQ(x)\in \Z$ for all $x\in A$}\}
\]
is called the {\em level of $\fqm A$}.

The metaplectic extension $\Mp_2(\Z)$ \label{bi1} of $\SL_2(\Z)$
(i.e.~the nontrivial twofold central extension of $\SL_2(\Z)$) can be
realized as the group of pairs $(M,\phi(\tau))$, where
$M=\kabcd\in\SL_2(\Z)$ and $\phi$ is a holomorphic function on the
complex upper half plane $\H$ with $\phi(\tau)^2=c\tau+d$ (see
e.~g.~\cite{Shimura}).  The group $\SL_2(\Z)$ is generated by
\begin{equation*}
  T=\zxz 1101\qquad\text{and}\qquad S=\zxz 0{-1}10
  ,
\end{equation*}
and the group $\Mp_2(\Z)$ is generated by $T^*:=(T,1)$ and
$S^*=(S,\sqrt \tau)$ with relations ${S^*}^2=(S^*T^*)^3=\zeta$, where
$\zeta=\left( \kzxz{-1}{0}{0}{-1}, i\right)$ is the standard generator
of the center of $\Mp_2(\Z)$.

The Weil representation $\rho_{\fqm A}$ associated to $\fqm A$ is a
representation of $\Mp_2(\Z)$ on the group algebra $\C[A]$. Here, and
throughout, we denote the standard basis of $\C[A]$ by
$(\frake_x)_{x \in A}$. The action of $\rho_{\fqm A}$ can
then be given in terms of the generators $S,T\in\Mp_2(\Z)$ as follows:
\begin{align*}
  \rho_{\fqm A}(T^*)\frake_x &= e(Q(x)) \frake_x,\\
  \rho_{\fqm A}(S^*)\frake_x &= \frac{ e(-\sig(\fqm A)/8)}{\sqrt{\cd A}}
                                    \sum_{y \in A} e\left(-Q(x,y)\right) \frake_y.
\end{align*}
We shall sometimes simply write $\alpha.v$ for
$\rho_{\fqm A}(\alpha)v$, i.e.~we consider $\C[A]$ as
$\Mp_2(\Z)$-module via the action
$(\alpha,v)\mapsto \rho_{\fqm A}(\alpha)v$. For details of the theory
of Weil representations attached to finite quadratic modules we refer
the reader to~\cite{hitchhikers}, \cite{Nobs},
\cite{Nobs-Wolfart},~\cite{Skoruppa},~\cite{Stroemberg}.

The kernel of the projection of $\Mp_2(\Z)$ onto its first coordinate
is the subgroup generated by~$(1,-1)$. It is easily checked that
$\rho_{\fqm A}((1,-1))=\rho_{\fqm A}(S^*)^4$ acts as multiplication by
$e(\sig(\fqm A)/2)$. This simple observation has two immediate
consequences. First of all, the space of invariants
$\C[A]^{\Mp_2(\Z)}$, i.e.~the subspace of elements $v$ in $\C[A]$
fixed by $\Mp_2(\Z)$, reduces to~$\{0\}$ unless $\sig(\fqm A)$ is
even. Secondly, $\rho_{\fqm A}$ descends to a representation of
$\SL_2(\Z)$ if and only~$\sig(\fqm A)$ is even. Note, that~$\sig(\fqm A)$
is always even if the level of~$\fqm A$ is odd as follows from
Milgram's formula.

\section{Invariants}
\label{sec:invariants}
Let $\fqm A=(A,Q)$ be a finite quadratic module of level $N$.
We shall assume in this section that $\sig(\fqm A)$
is even.  As we saw at the end of the last section the space
of invariants is otherwise zero. The representation $\rho_{\fqm A}$ then
descends to a representation of $\SL_2(\Z)$ and, even more,
factors through a representation of the finite group
$\Gamma(N) \bs \SL_2(\Z)$, i.e.~of the group
\begin{equation*}
  G_N:=\SL_2(\Z/N\Z).
\end{equation*}
We will denote this representation also by $\rho_{\fqm A}$.

An easy closed and explicit formula for the dimension of $\C[A]^{G_N}$ is
not known for general~$\fqm A$.  Of course, orthogonality of group
characters yields
\[
\dim \C[A]^{G_N} = \frac{1}{\cd {G_N}} \sum_{g \in {G_N}} \tr(\rho_{\fqm A}(g)).
\]
While it is therefore in principle possible to compute the dimension of
$\C[A]^{G_N}$, there are two obstructions in practice . First of all, the
size of the sum on the right can become very large. More precisely,
the number of conjugacy classes of $G_N$ is asymptotically equal to $N$
for increasing~$N$ (see~\cite[Tabelle~2]{Nobs}). Secondly, the
straight-forward formulas for $\tr(\rho_{\fqm A}(g))$ which follow
from explicit formulas for the matrix coefficients of
$\rho_{\fqm A}(g)$ involve trigonometric sums with about $\cd A^2$
many terms (see e.g.~\cite[Theorem 6.4]{Stroemberg})\footnote{However,
  in~\cite{hitchhikers} a much simpler formula is given, which
  expresses the traces of the Weil representations in terms of the
  natural invariants for the conjugacy classes of $\SL_2(\Z)$.}.

The following proposition implies that we can compute the invariants
or the dimension of the space of invariants ``locally'', i.e. for
every $p$-component of $A$ separately. For a given
prime~$p$, denote the $p$-subgroup of $A$ by $A_p$. It is quickly
verified that $\fqm A_p := (A_p,Q|_{A_p})$ is again a finite quadratic module.
Moreover, the decomposition $A=\bigoplus_{p\mid \cd A} A_p$ of $A$ as sum
over its $p$-subgroups $A_p$ induces an orthogonal direct sum decomposition of $\fqm A$.
We also decompose $G_N$ as a product
\[
  G \cong \prod_{p^\nu \Vert N} G_{p^\nu}
\]
with $G_{p^\nu} := \SL_2(\Z/p^{\nu}\Z)$ via the Chinese remainder theorem.
In this way $\bigotimes_p \C[A_p]$ becomes a $G_N$-module in the obvious way. 
For this, we note that the set of primes dividing $N$ is equal to the set of primes
dividing $\cd A$.

\begin{proposition}
  \label{prop:local-decomposition}
  Let $A=\bigoplus_{p\mid N} A_p$ be the decomposition of $A$ as sum
  over its $p$-subgroups $A_p$. Then
  $\frake_{\oplus_p a_p}\mapsto \otimes_p \frake_{a_p}$ defines via
  linear extension an isomorphism of $G$-modules
  \[
  \C[A]\xrightarrow{\cong} \bigotimes_{p \mid N} \C[A_p].
  \]
  Under this isomorphism we have
  \begin{equation*}
    \C[A]^{G_N} \cong \bigotimes_{p^\nu \Vert N} \C[A_p]^{G_{p^\nu}}.
  \end{equation*}
\end{proposition}

\begin{remark}
  The proposition implies in particular
\[
  \dim_\C \C[A]^{G_N} = \prod_{p^\nu \Vert N} \dim_\C \C[A_p]^{G_{p^\nu}}.
\]
\end{remark}

\begin{proof}[Proof of Proposition~\ref{prop:local-decomposition}]
  The given map clearly defines an isomorphism of complex vector
  spaces.  That this map commutes with the action of $G_N$, where $G_N$
  acts component-wise on the right-hand side, as described above, is
  easily checked using the formulas for the $S$ and $T$-action.  It
  follows that
  \begin{equation*}
    \sym{tr}(g, \C[A])=\prod_p \sym{tr}(g_p,\C[A_p])
  \end{equation*}
  for all $g=\otimes_p g_p$ in $G_N$, which implies, in particular, the
  second statement via orthogonality of group characters.
\end{proof}

A natural problem is to determine the field or ring of
definition\footnote{We say that {\em a subspace $V$ of $\C[A]$ is
    defined over the ring~$R$} if it possesses a basis whose elements
  are in $R[A]$.} of the space $\C[A]^{G_N}$. From the formulas defining
$\rho_{\fqm A}$, it is clear that $\C[A]^{G_N}$ is defined over the
cyclotomic field $K_N$\footnote{For this one needs that
  $e(-\sig(\fqm A)/8)/\sqrt {\cd A}$ is in $K_N$, which can be read off
  from Milgram's formula.}.  However, it turns out that the invariants
are in fact defined over the field of rational numbers, as we shall
see in a moment. This will allow us in
Section~\ref{sec:reduction-mod-p} to compute a basis for $\C[A]^{G_N}$ by
doing the computations in $\F_\ell[A]$ for suitable sufficiently large
primes~$\ell$.

\begin{theorem}
  \label{thm:ring-of-definition}
  The space $\C[A]^{G_N}$ is defined over $\Z$.
\end{theorem}

For the proof we need some preparations.

\begin{lemma}
  \label{lem:Gamma0-action}
  For any $g= \zxz ab0d$ in $\SL_2(\Z/N\Z)$ and $x$ in $A$, one has
  \begin{equation*}
    \rho_{\fqm A}(g)\frake_x
    =
    \chi_{\fqm A}(d)\, e\left(bdQ(x)\right)\frake_{dx},
  \end{equation*}
  where $\chi_{\fqm A}(d)=\sigma_d(w)/w$ with
  $w=\sum_{x\in A}e\left(Q(x)\right)$.
\end{lemma}

A careful analysis of $\chi_{\fqm A}$ yields
\begin{equation*}
  \chi_{\fqm A}(d)=
  \begin{cases}
    \leg d{\cd A} & \text{ if } \cd A \text{ is odd,}\\
    \leg d{\cd A} \leg {-4}d^{s}& \text { if } \cd A \text{ is even,}
  \end{cases}
\end{equation*}
where $s=\sig(\fqm A_2)/2$ (see e.g.~\cite[Lemma
3.9]{fredrik-weilfqm}). However, we shall not need this formula.

\begin{proof}[Proof of Lemma~\ref{lem:Gamma0-action}]
  Since $\zxz ab0d = \zxz a00d \zxz 1{bd}01$ and
  $\zxz 1{bd}01\frake_{x}=e(bdQ(x))\frake_{x}$ it suffices
  to consider the action of $\zxz a00d$.  For this we write
  \begin{equation*}
    \zxz a00d = S^{-1} \zxz 1{d}01 S \zxz 1{a}01 S \zxz 1{d}01
  \end{equation*}
  and apply the formulas for the action of $S$ and $T$ to obtain
  \begin{gather*}
    \zxz a00d \frake_{x} = \gamma\, \frake_{dx}, \intertext{where} \gamma
    = \sigma_d(w)/w , \qquad w=\sum_{x\in
      A}e\left(Q(x)\right)
    .
  \end{gather*}
  (Here we used Milgram's formula). This proves the lemma.
\end{proof}

For any $s$ in $(\Z/N\Z)^\times$, let $\sigma_s$ denote the
automorphism of $K_N$ which sends each $N$th root of unity $z$
to~$z^N$. For any endomorphism $f$ of $\C[A]$ which leaves invariant
$K_N[A]$, say $f\frake_x=\sum_{y\in A}f(x,y)\frake_y$ with $f(x,y)$
in~$K_N$, we use $\sigma_s(f)$ for the endomorphism of $\C[A]$ such
that
\begin{equation*}
  \sigma_s(f)\frake_x=\sum_{y\in A}\sigma_s\left(f(x,y)\right)\frake_y.
\end{equation*}
Note that $f\mapsto \sigma_s(f)$ defines an automorphism of the ring
of endomorphisms of $\C[A]$ which leave invariant $K_N[A]$.

\begin{lemma}
  \label{lem:Galois-intertwining}
  For any $\kzxz abcd$ in $G_N$, one has
  \begin{equation*}
    \sigma_s(\rho_{\fqm A}\left(\kzxz abcd\right))
    =
    \rho_{\fqm A}
    \left(\kzxz{a}{sb}{s^{-1}c}{d}\right)
    .
  \end{equation*}
\end{lemma}

\begin{proof}
  Both sides of the claimed identity are multiplicative in
  $\kzxz abcd$ (for this note that the map
  $\kzxz abcd\mapsto \kzxz{a}{sb}{s^{-1}c}{d}$ defines a automorphism
  of~$G_N$). It suffices therefore to prove the claimed formula for the
  generators $T$ and $S$ of $G_N$. For $T$ the formula can be read off
  immediately from the formula for the action of $T$. For $S$ one has
  on the one hand side for any $x$ in $A$
  \begin{equation*}
    \sigma_s(\rho_{\fqm A}(S))\frake_x
    =
    \sigma_s(w)\sum_{y\in A} e\left(-sQ(x,y)\right) \frake_y
    ,
  \end{equation*}
  where $w= {e(-\sig(\fqm A)/8)}/{\sqrt{\cd A}}=\sum_{x\in A}e(-Q(x))/\cd A$. On the other hand,
  $\kzxz 0{-s}{s^{-1}}0=S\kzxz {s^{-1}}00s$, and hence, using
  Lemma~\ref{lem:Gamma0-action},
  \begin{equation*}
    \rho_{\fqm A}\left(\kzxz 0{s}{s^{-1}}0\right)\frake_x
    =
    \chi_{\fqm A}(s) w \sum_{y\in A} e\left(-Q(sx,y)\right) \frake_y
    .
  \end{equation*}
  But $\sigma_s(w)/w = \chi_{\fqm A}(s)$, which implies the claimed
  formula.
\end{proof}

\begin{proof}[Proof of Theorem~\ref{thm:ring-of-definition}]
  The $G_N$-invariant projection $\calP: \C[A] \to \C[A]^{G_N}$ is given by
  the formula
  \begin{equation*}
    \calP = \frac{1}{\cd {G_N}} \sum_{g \in G_N} \rho_{\fqm A}(g)
    .
  \end{equation*}
  It suffices to show that, for any $x$ in~$A$, we have
  $\calP\frake_x=\sum_{y}\calP(x,y)\frake_y$ with rational numbers
  $\calP(x,y)$, in other words, that we have, for any $s$ in
  $(\Z/N\Z)^\times$ the identity $\sigma_s(\calP)=\calP$. But this
  follows from Lemma~\ref{lem:Galois-intertwining} and the fact that
  $\kzxz abcd \mapsto \kzxz{a}{sb}{s^{-1}c}{d}$ permutes the elements
  of~$G_N$. This proves the theorem.
\end{proof}

\section{The algorithm}
\label{sec:algorithm}

In this section we explain our algorithm for computing a basis for the
space of invariants. We then discuss various easy and natural
improvements. We fix a finite quadratic module $\fqm A=(A,Q)$ of
level~$N$, and assume that $\sig(\fqm A)$ is even (since otherwise the
space of invariants of the associated Weil representation is
trivial). The Weil representation~$\rho_{\fqm A}$ is then a
representation of $G = \SL_2(\Z)$, which factors even through a
representation of $G_N =\SL_2(\Z/N\Z)$. Define
\begin{equation*}
  \sym{Iso}(\fqm A):=\left\{x\in A: Q(x)=0\right\},
\end{equation*}
and, for $v \in \C[A]$,
\begin{equation*}
  \supp(v) := \{x \in A\ :\ v(x) \neq 0 \}.
\end{equation*}
Note that, for any $G$-submodule $M$ of $\C[A]$, we have
\begin{equation*}
  M^T:=\left\{v\in M: \rho_{\fqm A}(T)v=v\right\}
  =
  \left\{v\in M:\sym{supp}(v)\subseteq \sym{Iso}(\fqm A)\right\}
\end{equation*}
as follows immediately from the formula for the action of $T$ in
Section~\ref{sec:fqm}.  Our algorithm is based on the following
observation.

\begin{proposition}
  \label{prop:main-observation}
  Let $M$ be a $G$-submodule of $\C[A]$. Then
  \begin{equation*}
    M^G = \left(1+\rho_{\fqm A}(S)+\rho_{\fqm A}(S)^2+\rho_{\fqm A}(S)^3\right)(M^T)\cap M^T.
  \end{equation*}
\end{proposition}

\begin{proof}
  An element $v$ of $M$ is invariant under all of $G$ if it is
  invariant under the generators $T$ and $S$ of $G$, i.e.~if it is
  contained in $M^T$ and the set $M^S$ of vectors invariant under
  $\rho_{\fqm A}(S)$. Since $S^4=1$ we have $M^S=\sym{Tr}_S(M)$, where
  \begin{equation*}
    \sym{Tr}_S
    =
    1+\rho_{\fqm A}(S)+\rho_{\fqm A}(S)^2+\rho_{\fqm A}(S)^3
    .
  \end{equation*}
  But $M^G\subseteq M^T$, hence
  $M^G = \sym{Tr}_S(M^G)\subseteq\sym{Tr}_S(M^T)$, and therefore
  \begin{equation*}
    M^G=M^G\cap M^T\subseteq \sym{Tr}_S(M^T)\cap M^T
    .
  \end{equation*}
  The proposition is now obvious.
\end{proof}

The Proposition is quickly converted into a first version of our
algorithm:

\begin{algorithm}
  (Computing a basis for the space $\C[A]^G$ of invariants)
  \label{alg:algo-1}

  \begin{enumerate}
    \renewcommand\labelenumi{\arabic{enumi}.}
  \item Find the isotropic elements $a_1,\dots,a_m$ and the
    non-isotropic elements $b_1,\dots,b_n$ in $A$.
  \item Compute the $(m+n) \times m$ matrix $H$ such that
    \begin{equation*}
      (L\frake_{a_1},\dots,L\frake_{a_m})
      =
      (\frake_{a_1},\dots,\frake_{a_m},\frake_{b_1},\dots,\frake_{b_n})H
      ,
    \end{equation*}
    where
    $L=1+\rho_{\fqm A}(S)+\rho_{\fqm A}(S)^2+\rho_{\fqm A}(S)^3$.
  \item Let $U$ and $V$ be the matrices obtained by extraction the
    first $m$ and the last $n$ rows of $H$, respectively.
  \item Compute a basis $\mathfrak{V}$ for the space of vectors $x$
    such that $Vx=0$.
  \item Return a basis for the space of all $Ux$, where
    $x$ runs through the basis $\mathfrak{V}$ .
  \end{enumerate}

\end{algorithm}

For implementing this algorithm we need, first of all, to decide over
which field $K$ we would like to do the computations. One possibility
is to use floating point numbers to do a literal implementation using
the field of complex numbers. However, the matrix coefficients
of $\rho_{\fqm A}(S)$ with respect to the natural basis of $\C[A]$ are elements of the
$N$th cyclotomic field $K_N$. Hence it is reasonable to the
calculations over $K_N=\Q[x]/(\phi_N)$, where $\phi_N$ is the $N$th
cyclotomic polynomial. Another choice for $K$ will be discussed in
Section~\ref{sec:reduction-mod-p}.

There are two easy improvements which can help to reduce the computing
time. The first one is due to the following observation.
\begin{proposition}
  \label{prop:parity-restriction}
  The subspaces $\C[A]^+$ and $\C[A]^-$ of even and odd functions are
  $G$-submodules of $\C[A]$.  Let $\epsilon = (-1)^{\sig(\fqm A)/2}$.
  Then $\C[A]^G=(\C[A]^{\epsilon})^G$ and
  $(\C[A]^{-\epsilon})^G=\{0\}$.
\end{proposition}

\begin{proof}
  The first statement follows immediately from the observation that
  the map $\frake_a\mapsto \frake_{-a}$ intertwines with the action
  of~$S$ and $T$, and hence with the action of~$G$, as is obvious from
  the formulas for the action of $S$ and $T$.

  For the proof of the second statement we note that
  $S^2\frake_a = \epsilon\frake_{-a}$ which is again an immediate
  consequence of the formula for the action of~$S$. In other words,
  any invariant~$v$ satisfies $v(a)=(S^2v)(a)=\epsilon v(-a)$ for all
  $a$ in~$A$.
\end{proof}

Let $\rho_{\fqm A}^\pm:G\rightarrow \GL\left(\C[A]^\pm\right)$
afforded by the $G$-modules~$\C[A]^\pm$. As we saw in the proof of the
preceding proposition $S^2$ acts on $\C[A]^\epsilon$
($\epsilon = (-1)^{\sig(\fqm A)/2}$) as identity,
i.e.~$\rho_{\fqm A}^\epsilon(S^2)=1$. Using this
Propositions~\eqref{prop:main-observation},~\eqref{prop:parity-restriction}
imply
\begin{equation*}
  \C[A]^G
  =
  \C[A]^\epsilon
  =
  \left\{v\in \left(1+\rho_{\fqm A}^\epsilon(S)\right)(\C[A]^\epsilon):
    \sym{supp}(v)\subseteq \sym{Iso}(\fqm A)\right\}
  .
\end{equation*}
A basis for $\C[A]^\epsilon$ is obtained by replacing in the standard
basis $\frake_a$ by
$\frake_a^\epsilon = \frac 12\left(\frake_a+\epsilon
  \frake_{-a}\right)$
and omitting all zeroes and all duplicated vectors.  This leads to the
following modified algorithm.
\begin{algorithm}
  (Modified algorithm for computing a basis for the space of
  invariants)
  \label{alg:algo-2}

  \begin{enumerate}
    \renewcommand\labelenumi{\arabic{enumi}.}
  \item As in Algorithm~\ref{alg:algo-1}.
  \item [2.a] Construct the basis $\frake_{a_i}^\epsilon$,
    $\frake_{b_j}^\epsilon$ ($1\le i\le m'$, $1\le j\le n'$) of
    $\C[A]^\epsilon$ obtained from the standard basis $\frake_{a_i}$,
    $\frake_{b_j}$ by (anti-)symmetrizing, suppressing zeroes and
    duplicates, and after possibly renumbering the~$a_i$ and~$b_j$.
  \item [2.b] Compute the $(m'+n') \times n'$ matrix $H'$ such that
    \begin{equation*}
      (L\frake_{a_1}^\epsilon,\dots,L\frake_{a_{m'}}^\epsilon)
      =
      (\frake_{a_1}^\epsilon,\dots,\frake_{a_{m'}}^\epsilon,\frake_{b_1}^\epsilon,\dots,\frake_{b_{n'}}^\epsilon)H'
      ,
    \end{equation*}
    where $L=1+\rho_{\fqm A}^\epsilon(S)$.
  \item [3.--5.]  As in Algorithm~\ref{alg:algo-1} with $H$, $m$, $n$
    replaced by $H'$, $m'$, $n'$.
  \end{enumerate}

\end{algorithm}

The dimension of $\C[A]^{\pm}$ equals
$\frac 12\left(\cd A+\cd {A[2]}\right)$, where $A[2]$ denotes the
subgroup of elements annihilated by ``multiplication by~$2$''. Note
that $A[2]=\{0\}$ if $\cd A$ is odd. Therefore the size of $H'$ is
about half of the size of~$H$ in Algorithm~\ref{alg:algo-1}. Also note
that $H'$ has entries in the totally real subfield $K_N^+$ of
$K_N$. This implies that $\C[A]^G$ is in fact defined over $K_N^+$ and
we can perform our computations over $K_N^+$ instead of~$K_N$.

  To implement the algorithm, we still need an explicit formula for the entries of the matrix $H' = (h_{ij})$,
  where $1 \leq i, j \leq m' + n'$. We just write $x_i = a_i$ for $1 \leq i \leq m'$ and
  $x_i = b_{i-m'}$ for $m' < i \leq m'+n'$ for the elements of $A$.
  By a straightforward calculation, we obtain
  \begin{align*}
    h_{ij} &= f_{i}^{-1} \left\langle\rho_{\fqm A}(S)\frake_{x_j}^\eps + \frake_{x_j}^\eps, \frake_{x_i}^\eps\right\rangle \\
             &=  \frac{e(-\sig(A)/8)}{2f_i\sqrt{\cd A}}(e(-Q(x_j,x_i)) +\eps e(Q(x_j,x_i))) + \delta_{i,j},
  \end{align*}
  where $f_i = \langle \frake_{x_i}^\eps,  \frake_{x_i}^\eps \rangle$ and $\langle \cdot, \cdot \rangle$ denotes the standard hermitean inner product on $\C[A]$ (conjugate-linear in the second component), such that $\langle \frake_x, \frake_y \rangle = \delta_{x,y}$.
 Note that $f_i = \frac{1}{2}$ if $x_i \neq -x_i$ and $f_i = 1$, otherwise.

Given a finite quadratic module the exact value of quantity
$\sig(\fqm A)$ is not immediately clear. For finding the $\epsilon$ of
the preceding proposition the following is helpful.
\begin{proposition}
  For odd $\cd A$ one has
  \begin{equation*}
    (-1)^{\sig(\fqm A)/2} = \leg {-1}{\cd A}
    .
  \end{equation*}    
\end{proposition}

\begin{proof}
  Indeed, directly from the formula for the $S$-action we obtain
  $S^2\frake_x = (-1)^{\sig(\fqm A)/2}\frake_{-x}$. On the other hand
  $S^2=\zxz {-1}00{-1}$, and therefore we obtain by Lemma~\ref{lem:Gamma0-action}
  that $S^2 \frake_x = \chi_{\fqm A}(-1)\frake_{-x}$. For odd $\cd A$ it is
  then easy to deduce from the formula of the lemma for
  $\chi_{\fqm A}$ that $\chi_{\fqm A}(-1)=\leg {-1}{\cd A}$ (see also
  the remark after Lemma~\ref{lem:Gamma0-action}).
\end{proof}

The second possible improvement is the factorization into local
components as explained in
Proposition~\ref{prop:local-decomposition}. We compute first the local
components $\fqm A_p :=(A_p,Q|_{A_p})$, and apply then
Algorithm~\ref{alg:algo-2} to the finite quadratic modules $\fqm A_p$.
If the number of different primes in $\cd A$ is large this reduces the
run-timeof our algorithm enormously. Indeed, the two bottle necks of
our algorithm are the search for the isotropic elements in~$A$ and the
computation of the kernel of a matrix of size $m\times \cd A$, where
$m$ is the number of isotropic elements of $A$. If $\cd A$ contains
more than two different primes, say
$\cd A = p_1^{k_1}\cdots p_r^{k_r}$ with $r\ge 2$, then it takes
$p_1^{k_1}\cdots p_r^{k_r}$ many search steps to find all isotropic
elements in~$A$, whereas an application of
Proposition~\ref{prop:local-decomposition} allows us to dispense with
$p_1^{k_1}+\cdots +p_r^{k_r}$ many search steps to find eventually all
invariants of $A$. A similar comparison applies to the size of the
matrices in our algorithm when run either directly on $A$ or else
separately on the $p$-parts $A_{p_j}$.

\section{Reduction mod $\ell$}
\label{sec:reduction-mod-p}
In this section we fix again a finite quadratic module $\fqm A=(A,Q)$
of level $N$. Let $\ell$ denote a prime such that
$\ell\equiv 1\bmod N$. Then $\Q_\ell$ contains the $N$th roots of
unity, hence the $N$th cyclotomic field. Accordingly, we can
consider~$\rho_{\fqm A}$ as a representation of $G_N = \SL_2(\Z/N\Z)$
taking values in $\GL(\Q_\ell[A])$, and $\Q_\ell[A]$ as
$G_N$-module. From the formulas for the action of $S$ and $T$ on
$\Q_\ell[A]$ it is clear that $\Z_\ell[A]$ is invariant under~$G_N$, and
that the $\Z_\ell$-rank of $\Z_\ell[A]^{G_N}$ equals the dimension of
$\C[A]^{G_N}$.

For computing the rank of $\Z_\ell[A]^{G_N}$ it is natural to consider the
reduction modulo~$\ell$ of $\Z_\ell[A]$. More precisely, note that
$\ell\Z_\ell[A]$ is a $G_N$-submodule of $\Z_\ell[A]$, so that we have
the exact sequence of $G_N$-modules
\begin{equation*}
  0\longrightarrow
  \ell\Z_\ell[A]
  \longrightarrow
  \Z_\ell[A]
  \xrightarrow{\ r\ }
  \F_\ell[A]
  \longrightarrow 0,
\end{equation*}
where $r$ denotes the reduction map $r(f):a \mapsto f(a)+\ell\Z_\ell$.
Here the action of $G_N$ on $\F_\ell[A]\cong \Z_\ell[A]/\ell\Z_\ell[A]$
is the one induced by the action on~$\Z_\ell[A]$.

\begin{theorem}
  \label{thm:reduction-theorem}
  Suppose that $(N,\ell) \neq (2,3)$. Then
  \begin{equation*}
    \dim_{\Q_\ell} \Q_{\ell}[A]^{G_N} = \dim_{\F_\ell} \F_{\ell}[A]^{G_N}.
  \end{equation*}
\end{theorem}

\begin{remark}
  Numerical computed examples suggest that the theorem is also true for $N=2$ and $\ell = 3$.
  However, we did not try to pursue this further.
\end{remark}

\begin{proof}[Proof of Theorem~\ref{thm:reduction-theorem}]
  From the short exact sequence preceding the theorem we obtain the
  long exact sequence in cohomology
  \begin{equation*}
    0\longrightarrow
    \ell\Z_\ell[A]^{G_N}
    \longrightarrow
    \Z_\ell[A]^{G_N}
    \xrightarrow{\ r\ }
    \F_\ell[A]^{G_N}
    \longrightarrow
    \sym{H}^1(G_N,\ell\Z_\ell[A])
    \longrightarrow \dots.
  \end{equation*}
  We shall show in a moment that the order of $G_N$ is a unit of $\Z_\ell$. Hence, the
  cohomology group $\sym{H}^1(G_N,\ell\Z_\ell[A])$ is trivial~\cite[Corollary 10.2]{BrownCohom}.
  It follows then that $\F_\ell[A]^{G_N}\cong \Z_\ell[A]^{G_N}/\ell\Z_\ell[A]^{G_N}$. Since
  $\Z_\ell[A]^{G_N}$ is free we conclude that $\dim_{\F_\ell} \F_\ell[A]^{G_N}$ equals the
  $\Z_\ell$-rank of $\Z_\ell[A]^{G_N}$, which implies the proposition.

  For proving that $\cd G_N$ is not divisible by $\ell$, first note that $\ell \equiv 1 \bmod{N}$ implies that $\ell > N$.
  Then, recall that the order of $G_N = \SL_2(\Z/N\Z)$ is given by
  \[
    \cd G = N^3 \prod_{p \mid N}\frac{p^2-1}{p^2}.
  \]
  Hence, if $\ell \mid \cd G$, we have that there is a prime $p \mid N$, such that $\ell \mid p+1$ or $\ell \mid p-1$. 
  However, $p-1 < N < \ell$ and thus the only possibility is $\ell = p+1$ and $N=p$. Since $\ell$ and $p$ are primes we conclude $N=2$ and $\ell = 3$, which we excluded in the statement of the proposition.
\end{proof}

The results on reduction modulo $\ell$ are not only interesting from a
theoretical point of view.  Our implementation profits tremedously
from reduction modulo a suitable prime $\ell$ as it speeds up the
calculation in practice. The reason is that there are higly optimized
libraries for computation with matrices over finite fields (and/or
over the integers) available.  For instance, in \texttt{sage} (which
uses the linbox library default), computing the nullity of a random
$200 \times 200$ matrix with entries in a cyclotomic field
$\Q(\zeta_{11})$ takes about $4$ seconds on our test machine, whereas
computing the nullity of a $2000 \times 2000$ matrix over $\F_{23}$
takes about $600$ milliseconds.  This immediately speeds up the
computation of the dimension of $\C[A]^G$ although it does not give a
basis for $\C[A]^G$.

\section{Tables}
\label{sec:tables}

  Tables~\ref{tab:21} to~\ref{tab:52} list the values
  $s = \sig(\fqm A)$ and
  dimension $d = \dim \C[A]^{\SL_2(\Z)}$ for various $p$-modules
  $\fqm A=(A,Q)$, where $p=2,3,5$. We use {\em genus symbols} for
  denoting isomorphism classes of finite quadratic modules
  (see~\cite{Skoruppa, bef-simple}).
  In short, for a power $q$ of an odd
  prime $p$ and a nonzero integer $d$ the symbol $q^{d}$ stands for
  the quadratic module
  \begin{equation*}
    \left(\left(\Z/q\Z\right)^{k}, \frac {x_1^2+\cdots+x_{k-1}^2+ax_k^2}q\right)
    ,
  \end{equation*}
  where $k=|d|$ and $a$ is an integer such that
  $\leg {2a}p = \sym{sign}(d)$.  For a $2$-power $q=2^e$, we have the following symbols:
  We write $q^{d}_{a}$ for the module
  \begin{equation*}
    \left(\left(\Z/q\Z\right)^{k}, \frac {x_1^2+\cdots+x_{k-1}^2+a x_k^2}{2q}\right),
  \end{equation*}
  with $k = |d|$ and $\leg{a}{2} = \sym{sign}(d)$.
  We normalize $a$ to be contained in the set $\{1,3,5,7\}$ and if $q = 2$, we take $a \in \{1,7\}$.
  Finally, we write $q^{+2k}$ for
  \begin{equation*}
    \left(\left(\Z/q\Z\right)^{2k}, \frac {x_1x_2+\cdots+x_{k-1}x_k}q\right),
  \end{equation*}
  and $q^{-2k}$ for
  \begin{equation*}
    \left(\left(\Z/q\Z\right)^{2k}, \frac {x_1x_2+\cdots+x_{k-3}x_{k-2} + x_{k-1}^2 + 2x_{k-1}x_{k} + x_k^2}q\right).
  \end{equation*}

  The concatenation of such symbols stands for the direct sum of the
  corresponding modules. For instance, $3^{-1}9^{+1}27^{-2}$ denotes the finite quadratic module
  \begin{equation*}
    \left(\Z/3\Z\times \Z/9\Z\times(\Z/27\Z)^2,\frac {x^2}3-\frac {y^2}9+\frac {z^2-w^2}{27}\right).
  \end{equation*}
  It can be shown that every finite quadratic
  $p$-module is isomorphic to a module which can be described by such
  symbols, and that this description is essentially unique (up to some
  ambiguities for $p=2$). For details of this we refer to~\cite{Skoruppa}.

  For the computations we used~\cite{Sage}, the additional
  package~\cite{fqm-p} and our implementation of Algorithm \ref{alg:algo-2},
  which is available as part of the package \cite{sfqm-p}.

\begin{table}[h]
  \caption{Dimension $d = \dim_\C \C[A]^G$ for some $2$-modules of even signature $s$\label{tab:21}}
  \begin{tabularx}{\linewidth}{lX |lX |lX |lXX} \toprule
$\fqm A$ & $d$ & $\fqm A$ & $d$ & $\fqm A$ & $d$ & $\fqm A$ & $d$ & $s$ \\
$s=0$ &  & $s=4$ &  & $s=0$ &  &  &  &  \\
\midrule
 $2^{+2}$ & 2 & $2^{-2}$ & 0 & $4^{+2}$ & 3 & $4^{-8}$ & 1191 & 0 \\
 $2^{+4}$ & 5 & $2^{-4}$ & 1 & $4^{+4}$ & 16 & $2_0^{+2}$ & 1 & 0 \\
 $2^{+6}$ & 15 & $2^{-6}$ & 7 & $4^{+6}$ & 141 & $2_2^{+2}$ & 0 & 2 \\
 $2^{+8}$ & 51 & $2^{-8}$ & 35 & $4^{+8}$ & 1711 & $2_0^{+4}$ & 2 & 0 \\
 $2^{+10}$ & 187 & $2^{-10}$ & 155 & $4^{-2}$ & 1 & $2_4^{+4}$  & 0 & 4 \\
 $2^{+12}$ & 715 & $2^{-12}$ & 651 & $4^{-4}$ & 6 & $2_6^{+6}$ & 0 & 6 \\
 $2^{+14}$ & $2795$  &  $2^{-14}$ & $2667$  & $4^{-6}$ & 73 & $2_0^{+6}$  & 5  & 0  \\
 \bottomrule
  \end{tabularx}\vspace{2mm}
\end{table}

\begin{table}[h]
  \caption{Dimension $d = \dim_\C \C[A]^G$ for some $2$-modules of even signature $s$\label{tab:22}}
  \begin{tabularx}{\linewidth}{lX |lX |lX} \toprule
$\fqm A$ & $d$ & $\fqm A$ & $d$ & $\fqm A$ & $d$ \\
$s=0$ &  & $s=4$ &  & $s=2$ &  \\
\midrule
 $2^{+2}4^{+2}$ & 8 & $2^{+2}8^{-2}$ & 1 & $2^{+4}4_2^{+2}$ & 0 \\
 $2_0^{+2}4^{+2}$ & 4 & $2_0^{+2}8^{-2}$ & 0 & $2_0^{+4}4_2^{+2}$ & 0 \\
 $2^{+4}4^{+2}$ & 25 & $2^{+4}8^{-2}$ & 7 & $4_2^{+4}$ & 1 \\
 $2^{+2}8^{+2}$ & 11 & $2_0^{+4}8^{-2}$ & 1 & $2^{+2}4_2^{+4}$ & 4 \\
 $2_0^{+4}4^{+2}$ & 11 & $4^{+2}8^{-2}$ & 2 & $2_0^{+2}4_2^{+4}$ & 4 \\
 $2_7^{+1}4^{+2}8_1^{+1}$ & 4 & $4_0^{+2}8^{-2}$ & 4 & $4_2^{+2}8^{+2}$ & 3 \\
 $2_5^{-1}4^{+2}8_3^{-1}$ & 4 & $2^{+2}4^{-2}8^{-2}$ & 1 & $4_1^{+3}16_1^{+1}$ & 1 \\
 $2_0^{+2}8^{+2}$ & 6 & $2_0^{+2}4^{-2}8^{-2}$ & 3 & $4_7^{-3}16_3^{-1}$ & 1 \\
 \bottomrule
  \end{tabularx}\vspace{2mm}
\end{table}

\begin{table}[h]
  \caption{Dimension $d = \dim_\C \C[A]^G$ for some $3$-modules of signature $s$\label{tab:31}}
  \begin{tabularx}{\linewidth}{lX |lX |lX |lXX} \toprule
$\fqm A$ & $d$ & $\fqm A$ & $d$ & $\fqm A$ & $d$ & $\fqm A$ & $d$ & $s$ \\
$s=6$ &  & $s=2$ &  & $s=0$ &  &  &  &  \\
\midrule
 $3^{+1}$ & 0 & $3^{-1}$ & 0 & $9^{+1}$ & 1 & $27^{+1}$ & 0 & 6 \\
 $3^{-2}$ & 2 & $3^{+2}$ & 0 & $9^{+2}$ & 1 & $27^{+2}$ & 0 & 4 \\
 $3^{+3}$ & 1 & $3^{-3}$ & 1 & $9^{+3}$ & 5 & $27^{+3}$ & 5 & 2 \\
 $3^{-4}$ & 1 & $3^{+4}$ & 7 & $9^{+4}$ & 33 & $27^{-1}$ & 0 & 2 \\
 $3^{+5}$ & 10 & $3^{-5}$ & 10 & $9^{+5}$ & 121 & $27^{-2}$ & 4 & 0 \\
 $3^{-6}$ & 40 & $3^{+6}$ & 22 & $9^{-1}$ & 1 & $27^{-3}$ & 5 & 6 \\
 $3^{+7}$ & 91 & $3^{-7}$ & 91 & $9^{-2}$ & 3 & $81^{+1}$ & 1 & 0 \\
 $3^{-8}$ & 247 & $3^{+8}$ & 301 & $9^{-3}$ & 5 & $81^{+2}$ & 1 & 0 \\
 $3^{+9}$ & 820 & $3^{-9}$ & 820 & $9^{-4}$ & 11 & $81^{-1}$ & 1 & 0 \\
 $3^{-10}$ & 2542 & $3^{+10}$ & 2380 & $9^{-5}$ & 121 & $81^{-2}$ & 5 & 0 \\
 \bottomrule
  \end{tabularx}\vspace{2mm}
\end{table}

\begin{table}[h]
  \caption{Dimension $d = \dim_\C \C[A]^G$ for some $3$-modules of signature $s$\label{tab:32}}
  \begin{tabularx}{\linewidth}{lX |lX |lX |lX} \toprule
$\fqm A$ & $d$ & $\fqm A$ & $d$ & $\fqm A$ & $d$ & $\fqm A$ & $d$ \\
$s=6$ &  & $s=2$ &  & $s=6$ &  & $s=2$ &  \\
\midrule
 $3^{+1}27^{-1}$ & 2 & $3^{+1}27^{+1}$ & 0 & $3^{+1}243^{-1}$ & 2 & $3^{+1}243^{+1}$ & 0 \\
 $3^{-2}27^{-1}$ & 1 & $3^{-2}27^{+1}$ & 1 & $3^{-2}243^{-1}$ & 1 & $3^{-2}243^{+1}$ & 1 \\
 $3^{+3}27^{-1}$ & 1 & $3^{+3}27^{+1}$ & 7 & $3^{+3}243^{-1}$ & 1 & $3^{+3}243^{+1}$ & 7 \\
 $3^{-4}27^{-1}$ & 10 & $3^{-4}27^{+1}$ & 10 & $3^{-4}243^{-1}$ & 10 & $3^{-4}243^{+1}$ & 10 \\
 $3^{+5}27^{-1}$ & 40 & $3^{+5}27^{+1}$ & 22 & $3^{+5}243^{-1}$ & 40 & $3^{+5}243^{+1}$ & 22 \\
 $3^{-1}27^{+1}$ & 2 & $3^{-1}27^{-1}$ & 0 & $3^{-1}243^{+1}$ & 2 & $3^{-1}243^{-1}$ & 0 \\
 $3^{+2}27^{+1}$ & 1 & $3^{+2}27^{-1}$ & 1 & $3^{+2}243^{+1}$ & 1 & $3^{+2}243^{-1}$ & 1 \\
 $3^{-3}27^{+1}$ & 1 & $3^{-3}27^{-1}$ & 7 & $3^{-3}243^{+1}$ & 1 & $3^{-3}243^{-1}$ & 7 \\
 $3^{+4}27^{+1}$ & 10 & $3^{+4}27^{-1}$ & 10 & $3^{+4}243^{+1}$ & 10 & $3^{+4}243^{-1}$ & 10 \\
 $3^{-5}27^{+1}$ & 40 & $3^{-5}27^{-1}$ & 22 & $3^{-5}243^{+1}$ & 40 & $3^{-5}243^{-1}$ & 22 \\
 \bottomrule\
  \end{tabularx}\vspace{2mm}\
\end{table}

\begin{table}[h]
  \caption{Dimension $d = \dim_\C \C[A]^G$ for some $5$-modules of signature $s$\label{tab:51}}
  \begin{tabularx}{\linewidth}{lX |lX |lX |lX} \toprule
$\fqm A$ & $d$ & $\fqm A$ & $d$ & $\fqm A$ & $d$ & $\fqm A$ & $d$ \\
$s=4$ &  & $s=0$ &  & $s=4$ &  & $s=0$ &  \\
\midrule
 $5^{+1}$ & 0 & $5^{-1}$ & 0 & $25^{+1}$ & 1 & $125^{+1}$ & 0 \\
 $5^{-2}$ & 0 & $5^{+2}$ & 2 & $25^{+2}$ & 3 & $125^{+2}$ & 4 \\
 $5^{+3}$ & 1 & $5^{-3}$ & 1 & $25^{+3}$ & 7 & $125^{-1}$ & 0 \\
 $5^{-4}$ & 1 & $5^{+4}$ & 11 & $25^{-1}$ & 1 & $125^{-2}$ & 0 \\
 $5^{+5}$ & 26 & $5^{-5}$ & 26 & $25^{-2}$ & 1 &  &  \\
 $5^{-6}$ & 106 & $5^{+6}$ & 156 & $25^{-3}$ & 7 &  &  \\
 $5^{+7}$ & 651 & $5^{-7}$ & 651 &  &  &  &  \\
 \bottomrule
  \end{tabularx}\vspace{2mm}
\end{table}

\begin{table}[h]
  \caption{Dimension $d = \dim_\C \C[A]^G$ for some $5$-modules of signature $s$\label{tab:52}}
  \begin{tabularx}{\linewidth}{lX |lX |lX |lX} \toprule
$\fqm A$ & $d$ & $\fqm A$ & $d$ & $\fqm A$ & $d$ & $\fqm A$ & $d$ \\
$s=4$ &  & $s=0$ &  & $s=4$ &  & $s=0$ &  \\
\midrule
 $5^{+1}125^{-1}$ & 0 & $5^{-1}125^{+1}$ & 0 & $5^{+1}125^{+1}$ & 2 & $5^{-1}125^{-1}$ & 2 \\
 $5^{-2}125^{-1}$ & 1 & $5^{+2}125^{+1}$ & 1 & $5^{-2}125^{+1}$ & 1 & $5^{+2}125^{-1}$ & 1 \\
 $5^{+3}125^{-1}$ & 1 & $5^{-3}125^{+1}$ & 1 & $5^{+3}125^{+1}$ & 11 & $5^{-3}125^{-1}$ & 11 \\
 $5^{-4}125^{-1}$ & 26 & $5^{+4}125^{+1}$ & 26 & $5^{-4}125^{+1}$ & 26 & $5^{+4}125^{-1}$ & 26 \\
 \bottomrule
  \end{tabularx}\vspace{2mm}
\end{table}

\FloatBarrier

\bibliographystyle{alpha}
\bibliography{bib}

\newcommand{\etalchar}[1]{$^{#1}$}
\def\cprime{$'$}
\begin{thebibliography}{BEF16}

\bibitem[BEF16]{bef-simple}
Jan~Hendrik Bruinier, Stephan Ehlen, and Eberhard Freitag.
\newblock Lattices with many {B}orcherds products.
\newblock {\em Math. Comp.}, 85(300):1953--1981, 2016.

\bibitem[Bro82]{BrownCohom}
Kenneth~S Brown.
\newblock {Cohomology of Groups}, 1982.

\bibitem[BS17]{hitchhikers}
Hatice Boylan and Nils-Peter Skoruppa.
\newblock Explicit formulas for {W}eil representations of {$\sym{SL}(2)$}.
\newblock preprint, 2017.

\bibitem[Ehl16]{sfqm-p}
Stephan Ehlen.
\newblock {\em {F}inite quadratic modules and simple lattices, source code and
  resources, version 0.2}, 2016.
\newblock \url{http://www.github.com/sehlen/sfqm}.

\bibitem[MH73]{Milnor-Husemoeller}
John Milnor and Dale Husemoller.
\newblock {\em Symmetric bilinear forms}.
\newblock Springer-Verlag, New York, 1973.
\newblock Ergebnisse der Mathematik und ihrer Grenzgebiete, Band 73.

\bibitem[Nik79]{Nikulin}
V.~V. Nikulin.
\newblock Integer symmetric bilinear forms and some of their geometric
  applications.
\newblock {\em Izv. Akad. Nauk SSSR Ser. Mat.}, 43(1):111--177, 238, 1979.

\bibitem[Nob76]{Nobs}
Alexandre Nobs.
\newblock Die irreduziblen {D}arstellungen der {G}ruppen {$SL_{2}(Z_{p})$},
  insbesondere {$SL_{2}(Z_{2})$}. {I}.
\newblock {\em Comment. Math. Helv.}, 51(4):465--489, 1976.

\bibitem[NW76]{Nobs-Wolfart}
Alexandre Nobs and J{\"u}rgen Wolfart.
\newblock Die irreduziblen {D}arstellungen der {G}ruppen {$SL_{2}(Z_{p})$},
  insbesondere {$SL_{2}(Z_{p})$}. {II}.
\newblock {\em Comment. Math. Helv.}, 51(4):491--526, 1976.

\bibitem[S{\etalchar{+}}13]{Sage}
W.\thinspace{}A. Stein et~al.
\newblock {\em {S}age {M}athematics {S}oftware ({V}ersion 5.8)}.
\newblock The Sage Development Team, 2013.
\newblock {\url{http://www.sagemath.org}}.

\bibitem[S{\etalchar{+}}16]{fqm-p}
N.\thinspace{}P. Skoruppa et~al.
\newblock {\em {F}inite {Q}uadratic {M}odules {P}ackage ({V}ersion 1.0)}.
\newblock The Countnumber Team, 2016.
\newblock {\url{http://data.countnumber.de}}.

\bibitem[Shi73]{Shimura}
Goro Shimura.
\newblock On modular forms of half integral weight.
\newblock {\em Ann. of Math. (2)}, 97:440--481, 1973.

\bibitem[Sko16]{Skoruppa}
Nils-Peter Skoruppa.
\newblock Weil representations associated to finite quadratic modules and
  vector-valued modular forms.
\newblock preprint, 2016.

\bibitem[Str13a]{Stroemberg}
Fredrik Str{\"o}mberg.
\newblock Weil representations associated with finite quadratic modules.
\newblock http://dx.doi.org/10.1007/s00209-013-1145-x, 2013.
\newblock Math. Z.

\bibitem[Str13b]{fredrik-weilfqm}
Fredrik Str\"{o}mberg.
\newblock Weil representations associated with finite quadratic modules.
\newblock {\em Mathematische Zeitschrift}, pages 1--19, 2013.

\bibitem[Wal63]{Wall-1}
C.~T.~C. Wall.
\newblock Quadratic forms on finite groups, and related topics.
\newblock {\em Topology}, 2:281--298, 1963.

\bibitem[Wal72]{Wall-2}
C.~T.~C. Wall.
\newblock Quadratic forms on finite groups. {II}.
\newblock {\em Bull. London Math. Soc.}, 4:156--160, 1972.

\end{thebibliography}
\end{document}